\documentclass[11pt,a4paper,reqno]{amsart}
\pdfoutput=1
\usepackage[a4paper,centering,left=2.7cm,right=2.7cm]{geometry}
\usepackage{amsfonts,amscd,amssymb,amsmath,amsthm,mathrsfs,xcolor,longtable,pbox}
\usepackage{hyperref}
\usepackage[thinlines]{easytable}
\usepackage{microtype}
\newcommand{\diag}[1]{\mbox{diag}\left\{#1\right\}}

\newcommand{\bb}[1]{\mathbb{#1}}

\newcommand{\C}{\bb{C}}
\newcommand{\Q}{\bb{Q}}
\newcommand{\R}{\bb{R}}
\newcommand{\Z}{\bb{Z}}
\newcommand{\N}{\bb{N}}

\newcommand{\G}{\Gamma}
\newtheorem{thm}{Theorem}[section]

\newtheorem{lemma}[thm]{Lemma}

\numberwithin{equation}{section}
\newcommand{\Sp}{\mathrm{Sp}}

\newcommand{\q}{\mathfrak{q}}

\newcommand{\Orth}{\mathrm{O}}
\newcommand{\SO}{\mathrm{SO}}

\newcommand{\GL}{\mathrm{GL}}
\newcommand{\SL}{\mathrm{SL}}
\newcommand{\SorO}{\text{(S)O}}
\newcommand{\transpose}{\mathrm{t}}
\DeclareMathOperator{\Mat}{Mat}

\newcommand{\Hilbert}[3]{\left(  #1 \:\!,\:\! #2 \right)_{#3}}
\newcommand{\sage}{\textsc{Sage}}
\newcommand{\gap}{\textsc{gap}}

\renewcommand{\geq}{\geqslant}
\renewcommand{\leq}{\leqslant}

\begin{document}

\date{\today}

\title[Commensurability of Hypergeometric groups]{Commensurability and arithmetic equivalence for orthogonal~hypergeometric~monodromy~groups}
\author{Jitendra Bajpai, Sandip Singh and Scott Thomson}
\address{J.~Bajpai: Mathematisches Institut, Georg-August Universit\"at G\"ottingen, Germany.} \email{Jitendra.Bajpai@mathematik.uni-goettingen.de} 
\address{S.~Singh: Department of Mathematics, Indian Institute of Technology Bombay, Mumbai, India.} \email{sandip@math.iitb.ac.in}
\address{S.~Thomson: Universit\"at Bern, Sidlerstrasse 5, CH-3012 Bern, Switzerland.}
\email{s.a.thomson@dunelm.org.uk}
\begin{abstract}
We compute invariants of quadratic forms associated to orthogonal hypergeometric groups of degree five. This allows us to determine some commensurabilities between these groups, as well as to say when some thin groups cannot be conjugate to each other.
\end{abstract}
\maketitle
%\tableofcontents

\section{Introduction}

The fundamental group $\pi_{1}$ of the sphere with three punctures is a free group on two generators.
This thrice-punctured sphere arises, for example, as $\mathbb{P}^1(\C)\backslash\{0,1,\infty\}$, on which can be defined the space $V$ of local solutions to the hypergeometric differential equation $D(\alpha, \beta)w=0$ (cf.~\eqref{introdifferentialequation}).
By analytic continuation of a basis for $V$ along a loop around each singularity in $\{0,1,\infty\}$, we obtain a linear action of $\pi_{1}$ on $V$ and hence a (not necessarily faithful) representation $\rho$ of $\pi_{1}$.
That is, we realise $\pi_{1}$ as a subgroup $\rho(\pi_{1})$ of $\GL_{n}(\C)$ for some $n$ depending on $D$ (see again below for a precise description of this construction).

\par Under certain conditions given by F.~Beukers and G.~Heckman \cite{BH89} the Zariski closure of $\rho(\pi_{1})$ is an orthogonal group of some  quadratic form $\q$.
Moreover under natural circumstances the representation $\rho(\pi_{1})$ is contained in $\GL_{n}(\Z)$ and the quadratic form $\q$ is defined over $\Z$.
It is therefore natural to investigate the properties of these quadratic forms as $\Q$-defined objects, and our investigations concern the similarity classes of these $\q$.

\par Associated to a rational quadratic form is a commensurability class of arithmetic groups, and since it is known in some cases that the monodromy groups $\rho(\pi_{1})$ described above are arithmetic subgroups of $\Orth(3,2)$, we are able to establish the commensurability of many pairs of monodromy groups as subgroups of $\Orth(3,2)$ by examining the similarity classes of their associated quadratic forms.

\par We work with a finite list of monodromy groups, which contains all the groups generated by the companion matrices of the pairs $\{f,g\}$ of polynomials of degree~$5$ that are products of cyclotomic polynomials with $f(0)/g(0)=-1$.
As explained below, this ratio condition on the constant term ensures that the Zariski closure of the group so generated is an orthogonal group.
For this finite list of orthogonal groups we compute invariants of the associated quadratic forms and so produce a classification of the monodromy groups in terms of their commensurabilities.
Where the groups are known to be \emph{arithmetic} we can directly say whether or not they are commensurable, and in the case where they are \emph{thin} we can say whether or not any conjugates can intersect in another thin group.

\par A commensurability classification is a natural one to have in the context of arithmetic lattices (especially since arithmeticity is in any case defined only up to commensurability).
For the groups of type $\Orth(3,2)$, we find just four similarity classes of quadratic form, and the groups known to be arithmetic are divided into three commensurability classes.
For the groups of type $\Orth(4,1)$ there are, likewise, four similarity classes, and for the groups of finite type ($\Orth(5)$) there are two, giving ten similarity classes of quadratic form in total.

\section{Hypergeometric monodromy groups}\label{sec:monodromygroups}
The monodromy groups of hypergeometric differential equations of order $n$ are defined as subgroups of $\GL_n(\C)$ that are the images of the fundamental group $\pi_1\left(\mathbb{P}^1(\C)\backslash\{0,1,\infty\}\right)$ under the monodromy representations of $\pi_1$ inside $\GL(V)$, where $V$ is the complex $n$-dimensional solution space about $z_{0}\notin\{0,1,\infty\}$ of the hypergeometric differential equation

\begin{equation}\label{introdifferentialequation}
D(\alpha;\beta)w=0
\end{equation}
on $\mathbb{P}^1(\C)$, having regular singularities at the points $\{0,1,\infty\}$, where the differential operator $D(\alpha;\beta)$ is defined by
\begin{align*}
D(\alpha;\beta)& =(\theta+\beta_1-1)\cdots(\theta+\beta_n-1)-z(\theta+\alpha_1)\cdots(\theta+\alpha_n),
\end{align*}
where $\theta=z\frac{d}{dz}$, and where $\alpha=(\alpha_1,\alpha_2,\ldots,\alpha_n)$, $\beta=(\beta_1,\beta_2,\ldots,\beta_n)\in\Q^n$. Note that the monodromy groups of the hypergeometric equations are defined up to conjugation in $\GL_n(\C)$, and they are also called the \emph{hypergeometric groups}.
(Henceforth the word `monodromy' will be omitted.)

\par The action of $\pi_{1}$ on $V$ is as follows: for a basis $\mathcal{B}=\{b_{1},\ldots,b_{n}\}$ of $V$, one may for each $b_{i}$ express its analytic continuation around the singularity $s\in\{0,1,\infty\}$ in terms of the elements of $\mathcal{B}$.
Each loop (element) $\gamma\in\pi_{1}$ thus defines by analytic continuation along $\gamma$ a linear combination of the $b_{j}$, for every $b_{i}\in\mathcal{B}$.
Hence $\gamma$ may be represented as a matrix $\overline{\gamma}$ in $\GL_{n}(\C)$ and the map $\rho\colon \pi_{1}\to\GL_{n}(\C)\colon \gamma \mapsto \overline{\gamma}$ is called the \emph{monodromy representation}.

The generators of the hypergeometric groups are determined by a theorem of Levelt (\cite{Levelt}; cf.~\cite[Theorem 3.5]{BH89}): if $\alpha_1,\alpha_2,\ldots,\alpha_n$, $\beta_1,\beta_2,\ldots,\beta_n\in\Q$ such that $\alpha_j-\beta_k\not\in\Z$, for all $j,k=1,2,\ldots,n$, then the hypergeometric groups are, up to conjugation in $\GL_n(\C)$, generated by the companion matrices $A$ and $B$ of the polynomials
\begin{equation}\label{parameterstopolynomials}
 f(X)=\prod_{j=1}^{n}(X-{{\rm e}^{2\pi i\alpha_j}})\quad\mbox{ and }\quad g(X)=\prod_{j=1}^{n}(X-{{\rm e}^{2\pi i\beta_j}})
\end{equation}
respectively, and the monodromy representation of $\pi_1$ is defined by sending $g_\infty$ to $A$, $g_0$ to $B^{-1}$, and $g_1$ to $A^{-1}B$, where $g_0, g_1, g_\infty$ are, respectively, the loops around $0,1,\infty$, which generate  $\pi_1$ modulo the relation $g_\infty g_1 g_0=1$.

We now denote the hypergeometric group by $\Gamma(f,g)$ (which is a subgroup of $\GL_n(\C)$ generated by the companion matrices $A$ and $B$ of the polynomials $f$ and $g$), and consider the cases where $f$ and $g$ are products of cyclotomic polynomials and $f(0)=\pm1$, $g(0)=\pm1$, that is, the respective companion matrices $A$ and $B$ are quasi-unipotent, and they belong to $\GL_n(\Z)$; and in these cases, $\Gamma(f,g)\subset\GL_n(\Z)$.
(Recall that an element of a ring is quasi-unipotent if some power is unipotent.)
We also assume that $f,g$ form a primitive pair in the sense of Beukers and Heckman \cite[Th.~5.3]{BH89} (cf.~Singh and Venkataramana \cite[p.592]{SiVe14}), and that they do not have any common root.
(To be a \emph{primitive pair} is to have \emph{no} polynomials $f_{1},g_{1}\in\Z[x]$ such that $f(x)=f_{1}(x^{k})$ and $g(x)=g_{1}(x^{k})$ for some $k\geq 2$ an integer.)\label{def:PrimitivePair}
 
Beukers and Heckman \cite[Theorem 6.5]{BH89} have completely determined the Zariski closures $G\subset\GL_n(\C)$ of the hypergeometric groups $\Gamma(f,g)$, which are briefly summarized as follows:

\begin{itemize}
\setlength{\itemsep}{2pt}
\item If $n$ is even, and $f(0)/g(0)=1$, then the hypergeometric group $\Gamma(f,g)$ preserves a non-degenerate integral symplectic form $\Omega$ on $\Z^n$, and $\Gamma(f,g)\subset\Sp_\Omega(\Z)$ is Zariski dense, that is, $G=\Sp_\Omega$.

\item If $\Gamma(f,g)$ is {\it infinite}, and $f(0) / g(0) = -1$, then $\Gamma(f,g)$ preserves a non-degenerate integral quadratic form $\q$ on $\Z^n$, and $\Gamma(f,g)\subset\mathrm{O}_\q(\Z)$ is Zariski dense, that is, $G=\mathrm{O}_\q$.

\item It follows from \cite[Corollary 4.7]{BH89} that $\Gamma(f,g)$ is {\it finite} if and only if either $\alpha_1<\beta_1<\alpha_2<\beta_2<\cdots<\alpha_n<\beta_n$ or $\beta_1<\alpha_1<\beta_2<\alpha_2<\cdots<\beta_n<\alpha_n$; and in this case we say that the roots of $f$ and $g$ {\it interlace} on the unit circle.
\end{itemize}

It is clear that $\Gamma(f,g)\subset G(\Z)$; and we call $\Gamma(f,g)$ {\it arithmetic} if it is of finite index in $G(\Z)$, and {\it thin} otherwise.

There have been many articles (cf.~\cite{BS15, BT14, FMS14, Fu14, HvS, Si15S, Si15SE, Si15O, SiVe14, Ve14}) in the direction of answering the following question of Sarnak \cite{Sa14}: for which pairs of polynomials $f,g$ are the associated hypergeometric groups $\Gamma(f,g)$ arithmetic?

In \cite{BS15} Bajpai and Singh considered all possible pairs $f,g$ (up to `scalar shifts') of degree~$5$ integer-coefficient polynomials, having roots of unity as their roots, forming a primitive pair, and satisfying the conditions $f(0)=1$ and $g(0)=-1$. Note that, for a pair $f,g$ satisfying these conditions, the associated hypergeometric group $\Gamma(f,g)$ is either finite or preserves a non-degenerate integral quadratic form $\q$ on $\Z^5$, and $\Gamma(f,g)\subset\mathrm{O}_\q(\Z)$ is Zariski dense in the orthogonal group $\mathrm{O}_\q$ of the quadratic form $\q$.

In \cite{BS15} Bajpai and Singh show that there are $77$ pairs (cf.~\cite[Tables 1-7]  {BS15})   of degree~$5$ polynomials $f,g$ that satisfy the conditions of the last paragraph. It follows from \cite{BS15, Si15O}, and \cite{Ve14} that $37$  of these $77$ pairs correspond to arithmetic orthogonal hypergeometric groups of type $\Orth(3,2)$ (cf.~\cite[Tables 2-4]{BS15}), and the quadratic forms associated to $4$ of the remaining pairs are positive definite, and hence the associated hypergeometric groups are finite.
There are a further $7$ hypergeometric groups, which are of type $\Orth(4,1)$, and whose thin-ness follows from \cite{FMS14} (cf.~\cite[Table 1]{BS15}). 
The arithmeticity or thin-ness of the remaining $29$ pairs is not yet known. Among these 29 pairs $19$ pairs are of type $\mathrm{O}(3,2)$ (cf.~\cite[Table 6]{BS15}) and $10$ of these are of type $\mathrm{O}(4,1)$ (cf.~\cite[Table 7]{BS15}).

\par  This article is concerned with the $77$ pairs of polynomials, or more precisely with their associated quadratic forms.
By computing arithmetical invariants of these forms we are able to establish which forms are similar to one another, and for those similar pairs of forms with associated hypergeometric groups being arithmetic, we conclude that these groups are commensurable in the wide sense.

\section{Arithmetic groups}\label{main}

In what follows we will need the notion of commensurability. If $G$ is a group with subgroups $G_{1}$ and $G_{2}$, then $G_{1}$ and $G_{2}$ are together called \emph{commensurable} if their intersection has finite index in each of them.
The groups $G_{1}$ and $G_{2}$ are said to be \emph{commensurable in the wide sense} if some conjugate of $G_{1}$ in $G$ is commensurable with $G_{2}$.
In what follows we will usually mean commensurability in the wide sense.

Let $\q$ be a rational quadratic form in $n$ variables (where $n\geq 2$), and suppose that its signature over $\R$ is $(p,q)$. Choosing a basis for the $\Q$-vector space $\Q^{n}$ we may represent $\q$ by a symmetric matrix $Q$ with entries in $\Q$. Taking this explicit matrix point of view we define the group
\[ \Orth_{\q}(\R) = \bigl\{ g\in\GL_{n}(\R)  \mid g^{\transpose} Q g = Q\bigr\}, \]
and for any subring $R\subseteq \R$ we denote by $\Orth_{\q}(R)$ the subgroup $\Orth_{\q}(\R) \cap \GL_{n}(R)$ consisting of elements with entries in $R$ (and inverse also having entries in $R$).
In particular we have the group $\Orth_{\q}(\Z)$ of integral points. Usually we will work with the groups $\SO_{\q}(\R) = \Orth_{\q}(\R) \cap \SL_{n}(\R)$ and their corresponding subgroups $\SO_{\q}(R)$.

\par Choosing a different basis for $\Q^{n}$ leads to different realisations of $\SorO_{\q}(\R)$, differing by conjugation. Any two such representations of $\SorO_{\q}(\Z)$ will be commensurable in $\GL_{n}(\R)$ and so we can unambiguously refer to a (wide) commensurability class of subgroups of $\SorO_{\q}(\R)$, by abuse of notation denoted $\SorO_{\q}(\Z)$. Any subgroup $\Gamma < \SorO_{\q}(\R)$, that lies in this commensurability class, is called arithmetic.
That is, $\Gamma$ is arithmetic if it is commensurable to some conjugate of $\SO_{\q}(\Z)$ by an element $\gamma\in\SO_{\q}(\R)$.

\par Two quadratic forms $\q_1$ and $\q_2$ (on vector spaces $V_1$ and $V_2$ respectively) are called isometric, if there exists a vector space isomorphism $\phi\colon V_{1} \to V_{2}$ such that $\q_{2}\bigl( \phi(v)\bigr) = \q_{1}(v)$ for every $v\in V$.
More generally, the two forms are \emph{similar} if there exists a non-zero $\lambda$ in the base field such that $\q_1$ and $\lambda \q_2$ are isometric.

\par In this article quadratic forms are defined by their associated matrices, and so it will be interesting to know whether or not two given such matrices represent the same isometry class of quadratic form or not.
Note that whenever two matrices $Q_1$ and $Q_2$ represent the same quadratic form $\q$, there is a matrix $P$ in $M_{n}(\Q)$ with $P^{\transpose}Q_{1}P=Q_{2}$. From the above discussion we see that the associated arithmetic groups are in the same commensurability class. Thus it is sufficient, in order to establish two groups' commensurability, to show that their defining quadratic forms are isometric.

\par If $\q_1$ and $\q_2$ are similar then their associated arithmetic groups (i.e., $\SO_{\q_1}(\Z)$ and $\SO_{\q_2}(\Z)$) are commensurable, since there exists a $\Q$-isomorphism 
$\SO_{\q_1}(\Q)\to \SO_{\q_2}(\Q)$.
If, conversely, the associated arithmetic groups are commensurable (via some conjugation) then their intersection is a lattice in $\SO(p,q)$ and so is Zariski dense by Borel's density theorem \cite[Chap.~V]{Raghunathan:Discrete}.
Thus the conjugation is defined over $\Q$ and defines a $\Q$-isomorphism $\phi\colon\SO_{\q_1}(\Q)\to \SO_{\q_2}(\Q)$.
In this article we are concerned with quadratic forms of odd dimension, and in this case the existence of the isomorphism $\phi$ implies that the forms $\q_{1}$ and $\q_{2}$ are similar over $\Q$ \cite[Prop.~5.4]{Meyer:TotallyGeodesic}.
(It is in general true that for hyperbolic forms --- of signature $(n-1,1)$ --- their arithmetic groups are commensurable in the wide sense (i.e., via a conjugation) if and only if the two forms are similar \cite[2.6]{GPS87}.)

\par We note that one may speak of arithmetic subgroups of more general semisimple Lie groups, which are often defined over some number field $k$ different from $\Q$.
Here we will only need to consider orthogonal groups defined by quadratic forms over $\Q$ and so we merely direct the interested reader to the literature for a broader exposition \cite{Marg91}, \cite{Borel-1}, \cite{BH62}. 
The arithmetic groups constructed directly from quadratic forms are not the only type of arithmetic groups in orthogonal groups but again we will not have occasion to refer to these `second type' arithmetic groups here \cite[p.221]{GeometryII}.

\par In this article we show that some groups that arise as hypergeometric monodromy groups are commensurable.
The groups in question are arithmetic subgroups of $\SO_{\q}(\Z)$ for some integral quadratic forms $\q$ (with $\q$ varying), and by showing that some of these forms are pairwise isometric over $\Q$ we can establish that their associated arithmetic groups are commensurable.

\par In order to address isometry equivalence of quadratic forms we compute associated invariants, namely the discriminant, the real signature, and the Hasse-Witt invariants at each $p$-adic completion $\Q_{p}$ of $\Q$.

\par If $\q$ is a quadratic form on $\Q^{N}$ then a \emph{determinant} of $\q$ is by definition the determinant of a matrix $Q$ of $\q$ in some basis of $\Q^{N}$.
This is of course dependent on the basis chosen but is in fact defined modulo multiplication by a square in $\Q^{\times}$.
We therefore define the \emph{discriminant} of $\q$ to be the image of any determinant in $\Q^{\times}/(\Q^{\times})^{2}$.
It is a fundamental fact that any $\q$ as above may be put into a diagonal form 
\begin{equation}\label{eqn:diagonalform} \sum_{i=1}^{N} a_{i}x^{2} \end{equation}
over $\Q$ \cite[42:1]{OMeara63}.
If $\q$ is now regarded as a real quadratic form via the usual embedding $\Q\hookrightarrow\R$, then we can further put $\q$ into the form $\sum_{i=1}^{m_{1}} x_{i}^{2} - \sum_{j=1}^{m_{2}} x_{j}^{2} $ over $\R$, where $m_{1}+m_{2}=N$.
Then the pair $(m_{1},m_{2})$ is called the \emph{(real) signature} of $\q$ (and is uniquely determined by $\q$).
It is well-known that two (non-degenerate) quadratic forms $\q_1$ and $\q_2$ over $\R$ are isometric if and only if they have the same signature \cite[\S61]{OMeara63} \cite[Ch.~4]{Cassels:RQF}.

\par From the diagonal form \eqref{eqn:diagonalform} we construct further invariants of $\q$ that concern its behaviour over the other completions of $\Q$.
For a prime $p\in\N$ denote by $\Q_{p}$ the field obtained as the completion of $\Q$ with respect to the $p$-adic valuation on $\Q$.
For any $a,b\in\Q_{p}$ the Hilbert symbol $\Hilbert{a}{b}{p}$ is defined by
\begin{equation}\label{eqn:Hilbertsymbol} \Hilbert{a}{b}{p} = \begin{cases} +1 & \text{if }f(x,y,z)=ax^{2} + by^{2} - z^{2} \text{ has a root in }\Q_{p}^{3}\setminus\bigl\{(0,0,0)\bigr\} \: ; \\ -1 & \text{otherwise.} \end{cases}  \end{equation}
(Note that this has the property that $\Hilbert{a_{1}a_{2}}{b}{p} = \Hilbert{a_{1}}{b}{p}\Hilbert{a_{2}}{b}{p}$ \cite[p.42]{Cassels:RQF}.)
By regarding $\Z_{p}$ as the inverse limit of the inverse system $\{\Z/p^{m}\Z\}_{m}$, roots of the polynomial $f(x,y,z)$ can be found by restricting to a search for roots over $\Z/p^{m}\Z$ for each $m\geq 1$ \cite[Prop.~6, p.14]{Serre:CourseArithmetic}.
For $\q$  given in the diagonal form \eqref{eqn:diagonalform}, at each prime $p$ can be associated to $\q$ the product of Hilbert symbols, known as the Hasse-Witt invariant:
\begin{equation}\label{eqn:HasseWitt}
W_{p}(\q) = \prod_{i<j} \Hilbert{a_{i}}{a_{j}}{p}.
\end{equation}
One may show that $W_{p}(\q)$ has value $-1$ for only those primes $p$ appearing in the entries of the diagonalised quadratic form $\q$, and also possibly at $p=2$. \cite[p.42]{Cassels:RQF}.
Thus we only need to compute $W_{p}(\q)$ at those primes $p$ dividing these coefficients of the diagonalised form of $\q$, and at $2$\label{HasseWittPrimes}.

\par It turns out that an isometry class of non-degenerate quadratic forms over $\Q$ is given by their discriminant, their real signature, and the Hasse-Witt invariant for every prime $p$ \cite[Ch.~4 \& 6]{Cassels:RQF}.

\subsection{Similar quadratic forms}
We note the following:
\begin{lemma}\label{lem:NormaliseDiscriminant}
Choose $\Delta\in\Q^{\ast}$, and suppose that $\q$ is a quadratic form over $\Q$ of non-zero discriminant in $n$ variables, where $n$ is \emph{odd}.
Then $\q$ is similar to a form $\q'$ of discriminant $\Delta$.
\end{lemma}
\begin{proof}
Choosing a basis and representing $\q$ with a matrix $\Mat(\q)$, write $\lambda=\Delta/\det(\Mat(\q))$.
Let $\q'=\lambda\q$.
Then $\det(\Mat(\q'))= \lambda^{n}\det(\Mat(\q)) = \Bigl(\Delta^{n-1}/\det(\Mat(\q))^{n-1}\Bigr) \cdot \Delta \cong \Delta \pmod{ (\Q^{\ast})^{2}}$ (observing that $n-1$ is even).
Thus the discriminant of $\q'$ is equal to $\Delta$, as a class in $\Q^{\ast}/(\Q^{\ast})^{2}$.
\end{proof}

Furthermore:
\begin{lemma}\label{lem:HasseSymbolInvariant}
    Let $\q$ be a quadratic form over $\Q$ in $n$ variables, where $n$ is congruent to $1 \pmod{4}$.
    Let $p$ be a prime, and let $\lambda\in\Q^{\ast}$.
    Then $W_{p}(\q) = W_{p}(\lambda \q)$.
\end{lemma}
\begin{proof}
    We suppose without loss of generality that $\q$ may be represented in the diagonal form $(a_{1},\ldots,a_{n})$.
    First note that
    \begin{align*}
	\Hilbert{\lambda a_{i}}{\lambda a_{j}}{p} & = \Hilbert{\lambda}{\lambda a_{j}}{p}  \: \Hilbert{a_{i}}{\lambda a_{j}}{p}  \\
	& = \Hilbert{\lambda}{\lambda}{p} \: \Hilbert{\lambda}{a_{i}}{p} \: \Hilbert{\lambda}{a_{j}}{p}\: \Hilbert{a_{i}}{a_{j}}{p} .
    \end{align*}

    Inserting these terms into the expression for $W_{p}(\lambda\q)$ (cf.~\eqref{eqn:HasseWitt}), we have
   \begin{align*}
       W_{p}(\lambda\q) &= \Hilbert{\lambda}{\lambda}{p}^{\sum_{i=1}^{n-1}i} \: \prod_{i<j} \Hilbert{\lambda}{a_{i}a_{j}}{p} \Hilbert{a_{i}}{a_{j}}{p} \\
       &= \Hilbert{\lambda}{\lambda}{p}^{n(n-1)/2} \: \Hilbert{\lambda}{a_{1}^{n-1}a_{2}^{n-1}\cdots a_{n}^{n-1}}{p} \:  \prod_{i<j} \Hilbert{a_{i}}{a_{j}}{p} \\
       &= W_{p}(\q) \qquad\text{\small(since $(n-1)/2$ is even)}. \qedhere
   \end{align*}
\end{proof}

We will later compute invariants for a family of quadratic forms on $\Q^{5}$, each of which is determined up to a rational multiple.
By Lemma~\ref{lem:NormaliseDiscriminant}, this rational multiple can be chosen so as to normalise the forms to discriminant $\pm 1$ depending on their real signature.
Then by Lemma~\ref{lem:HasseSymbolInvariant}, $W_{p}(\q)$ at each prime $p$ will remain unchanged.
Moreover, since $W_{p}(\q)$ is an invariant of the isometry class of $\q$, we see that for a fixed signature in odd dimension the sequence $\bigl(W_{p}(\q)\bigr)_{p\text{ prime}}$ determines a similarity class of quadratic form.
For, if $\q_{1}$ and $\q_{2}$ have different Hasse-Witt invariants, then for every $\lambda\in \Q^{\ast}$, $\q_{1}$ and $\lambda \q_{2}$ must still have different Hasse-Witt invariants and there can be no isometry $\q_{1}\cong \lambda \q_{2}$; i.e., no similarity between $\q_{1}$ and $\q_{2}$.

\par Observe that if $n$ above were \emph{even}, then multiplication of the form $\q$ above by any $\lambda\in\Q^{\ast}$ would not change the square class of the resulting determinant, and hence would not change the discriminant.

\subsection{Establishing commensurabilities between hypergeometric groups}
As described in \S\ref{sec:monodromygroups}, in certain cases the Zariski closure of a hypergeometric group $\Gamma(f,g)$ is an orthogonal group of an integral quadratic form $\q$.
In this case the group $\Gamma(f,g)$ is contained in $\Orth_{\q}(\Z)$. If the inclusion is with finite index, then the group $\Gamma(f,g)$ is arithmetic. As described above, two arithmetic orthogonal groups in odd dimension are commensurable if and only if their defining quadratic forms are similar.
By computing the appropriate invariants of the quadratic forms determined by a pair of hypergeometric groups, we show that certain of these pairs belong to the same commensurability class.
This argument applies only in the case where it is already established that the hypergeometric groups are arithmetic, and so commensurable to their ambient arithmetic groups.
In the case of thin hypergeometric groups, one might be able to establish a virtual isomorphism between a given pair, but this likely would not follow directly from an arithmeticity argument.

\subsubsection{Quadratic form associated to the pair $(f,g)$}

 The quadratic forms $\q$ (unique up to scalar) preserved by $\G(f, g)$ can be computed explicitly\,.
 Let $f, g$ be a pair of monic polynomials of degree~5, that are products of cyclotomic polynomials, do not have any common root in $\C$, form a primitive pair (cf.~p.\pageref{def:PrimitivePair}), and satisfy the condition that $f(0) = -1 $ and $g(0) = 1$.
 Then, $\G(f, g)$ preserves a non-degenerate integral quadratic form $\q$ on $\Z^{5}$ and $\G(f, g)\subset O_{\q}(\Z)$ is Zariski dense by Theorem~6.5 in~\cite{BH89}.
 Let $A$ and $B$ be the companion matrices of $f$ and $g$ respectively, and let $C = A^{-1}\! B$.
 Following the notation and method explained by Singh~\cite{Si15O} (cf.~\cite{FMS14},\cite{Ve14}), let $e_1, e_2, e_3, e_4, e_5$ be the standard basis vectors of $\Q^{5}$ over $\Q$, and $v$ be the last column vector of $C-I$, where $I$ is the identity matrix.
 Then,  $Cv =-v$.
 Therefore, using the invariance of $\q$ under the action of $C$, we get that $v$ is $\q$-orthogonal to the vectors $e_1, e_2, e_3, e_4$ and $\q(v, e_5) \neq 0$ (since $\q$ is non-degenerate).
 We may now assume that $\q(v, e_5) = 1$. 
 Then the set $ \{v,Av,A^{2}v,A^{3}v,A^{4}v\}$ is linearly independent over $\Q$.
 Since $\q$ is invariant under the action of $A$, that is, $\q(A^{i} v,A^{j} v)$ = $\q(A^{i+1} v,A^{j+1} v)$, for any $i, j \in  \Z$, to determine the quadratic form $\q$ on $\Q^{5}$, it is enough to compute $\q(v,A^{j}v)$, for $j = 0, 1, 2, 3, 4 $.
 Also, since $v$ is $\q$-orthogonal to the vectors $e_1, e_2, e_3, e_4$ and $\q(v, e_5) = 1$ (say), we get that $\q(v,A^{j}v) $ is the coefficient of $e_5$ in $A^{j}v$\,.
Since the companion matrix $A$ (resp.~$B$) of $f$ (resp.~$g$) maps $e_i$ to
$e_{i+1}$ for $1 \leq i \leq 4$, to know the quadratic forms $\q$ preserved by the
orthogonal hypergeometric groups it is enough to know the scalars $\q(e_1, e_j)$ for $1 \leq j \leq 5$ (since $\q(e_i, e_j) = \q(Ae_i,Ae_j) = \q(e_{i+1}, e_{j+1})$ for $1 \leq i, j \leq 4$). 
For simplicity, by writing
$$a=\q(e_1, e_1);\quad  b=\q(e_1, e_2);\quad c=\q(e_1, e_3); \quad d=\q(e_1, e_4); \quad e=\q(e_1, e_5), $$
we can write the matrix form $Q$ of the quadratic form $\q$ associated to the group $\G(f,g)$ as follows: 

\begin{equation}\label{eqn:MatrixQ} Q = \begin{pmatrix} 
	a & b & c & d & e \\
	b & a & b & c & d \\
	c & b & a & b & c \\
	d & c & b & a & b \\
	e & d & c & b & a 
\end{pmatrix} .
\end{equation}
In the tables in \S\S\ref{sec:Inv}--\ref{sec:Inv-ft}, the first rows $\big(\q(e_1, e_1),\, \q(e_1, e_2),\, \q(e_1, e_3),\, \q(e_1, e_4),\, \q(e_1, e_5)\big)$ of these matrix forms of $\q$ are listed.

\subsubsection{Computation of Hasse invariants}
By way of example, let $\q$ be the quadratic form given, relative to the standard basis, by the matrix
\[ Q = \begin{pmatrix} 
  3 & 0 & -1 & 0 & -5\\
  0 & 3 & 0 & -1 & 0\\
 -1 & 0 & 3 & 0 & -1\\
  0 & -1 & 0 & 3 & 0\\
 -5 & 0 & -1 & 0 & 3
\end{pmatrix},   \] 
of determinant $-2^{9}$.
By applying the change of basis matrix
$$\textstyle \begin{pmatrix}
    1 & 0 & 1/3 & 0 & 2\\
    0 & 1 & 0 & 1/3 & 0\\
    0 & 0 & 1 & 0 & 1\\
    0 & 0 & 0 & 1 & 0\\
    0 & 0 & 0 & 0 & 1
\end{pmatrix}, $$
we obtain the diagonal form for $\q$ given by the $5\times 5$ matrix $Q'=\diag{ 3/2,\, 3/2,\, 4/3,\, 4/3,\, -4}$.
By removing squares, it is then clear that $\q$ is isometric to the form 
\[ \diag{ \frac{3}{2},\, \frac{3}{2},\, \frac{1}{3},\, \frac{1}{3},\, -1}, \]
which by abuse of notation we will continue to call $\q$.
Therefore the Hilbert symbols (cf.~\eqref{eqn:Hilbertsymbol}) can be computed as follows:
\begin{align*}\label{eqn:HilbertSymbolsComputation}
\Hilbert{3/2}{3/2}{2} &= -1 &\Hilbert{3/2}{1/3}{2} &=1  &\Hilbert{3/2}{-1}{2} &= -1 \\ 
\Hilbert{1/3}{1/3}{2} &= -1 &\Hilbert{1/3}{-1}{2}  &=1  &\Hilbert{-1}{-1}{2} &= -1 .\addtocounter{equation}{1}\tag{\theequation}
\end{align*}

Let us examine the first Hilbert symbol.
After multiplying by $4$ --- a square --- its value depends on the existence of non-zero roots in $\Q_{2}$ of the polynomial $f=6x^{2}+6y^{2} - z^{2}$ (cf.~\eqref{eqn:Hilbertsymbol}).
There is a nontrivial root in $\Q_{2}$ if and only if there is a primitive root in $\Z/2^{m}\Z$ (for every $m\geq 1$) of the polynomial $f_{(m)}$ obtained from $f$ by reducing its coefficients modulo $2^{m}$.
A primitive root $(x_{1},\ldots,x_{n})$ is one in which none of the $x_{i}$ are divisible by $p$ \cite[p.13--14]{Serre:CourseArithmetic}.
In the first case we therefore seek a solution to the equation $x^{2}+y^{2}=0$ over $\Z/2\Z$.
Here one quickly sees that the possible primitive solutions are 
\[ (0,0,1), \quad (1,1,0), \quad \text{and}\quad (1,1,1). \]
Any solution over $\Z/2^{m}\Z$ must reduce to one of these solutions modulo $2$.
One also sees quickly that, modulo $4$, the equation $f_{(4)}=0$ has the solution $(1,1,1)$.
However, modulo $8$, a solution congruent (modulo $2$) to $(1,1,\alpha)$, for $\alpha=0$ or $\alpha=1$, is not possible.
Neither is a solution congruent to $(0,0,1)$.
Therefore the equation does not have a solution over $\Z/8\Z$, and hence the Hilbert symbol has value $-1$, as stated in \eqref{eqn:HilbertSymbolsComputation}.

\par The Hilbert symbol $\Hilbert{3/2}{1/3}{2}$ has value $1$, for the the polynomial $f=9x^{2} + 2y^{2}-6z^{2}$ has a root $(0,1,1)$ modulo $2^{3}$ and this is also a root of $f_{(1)}$ and $f_{(2)}$ on projecting to $\Z/2\Z$ and $\Z/4\Z$.
We have $f=f_{(m)}$ for all $m\geq 3$ (abusing notation slightly) and so there is a non-trivial root of $f_{(m)}$ for every $m\geq 1$.
Computation of the other Hilbert symbols is a similar process.

\par Finally, the Hasse invariant at $p=2$ is simply the product of all values in \eqref{eqn:HilbertSymbolsComputation}, this product being equal to $+1$ (cf.~\eqref{eqn:HasseWitt}).

\subsubsection{Computation in practice}
It would of course be impractical to compute by hand all Hasse-Witt invariants as outlined above, for every quadratic form and every prime $p$ required.
Therefore we rely on the software \sage\  \cite{Sage} to quickly compute these invariants for our family of forms at any prescribed primes.
\sage\ can diagonalise forms over $\Q$.
As mentioned on p.~\pageref{HasseWittPrimes}, we need only compute the Hasse-Witt invariants at primes appearing in the entries of a diagonalised form of $\q$, and at $2$.
Since our list of forms is finite, we can produce a finite list of all primes appearing in any form on the list, in the diagonal forms given by \sage's \verb+rational_diagonal_form()+ algorithm.
It may be that we examine more primes than are necessary (indeed this appears to be the case) but this is at least sufficient and takes little time for our list.
The highest prime appearing is $149$ and so we have \sage\ compute, at all primes up to and including $149$, the Hasse-Witt invariants for each form $\q$.
It turns out, however, that the highest prime for which any value of $W_{p}(\q)$ is not equal to $1$ is $5$.
Thus we only list the values of $W_{p}(\q)$ for the first few primes, all others after this being equal to $1$.

\par In \S\S\ref{sec:Inv}--\ref{sec:Inv-ft} is given a list of quadratic forms whose matrices are of the form \eqref{eqn:MatrixQ}, so that only the first line needs to be given in each row of the table in order to deduce the quadratic form.
For each such vector of first-row entries we construct a quadratic form in \sage\ and ask it to compute the Hasse-Witt invariants at primes up to $p=149$:
\begin{verbatim}
def M(a,b,c,d,e): return
         matrix(ZZ,[[a,	b,	c,	d,	e],
                    [b,	a,	b,	c,	d],
                    [c,	b,	a,	b,	c],
                    [d,	c,	b,	a,	b],
                    [e,	d,	c,	b,	a]]);
Forms = [0 for i in range(NumberofForms)]
    Forms[0] = M(3, 0, -5, 0, 35)  
    Forms[1] = M(-67,-7,101,41,-547) 
    Forms[2] = M(-11,-3,13, 13, -43)
Q = [0 for i in range(NumberofForms) ]
D = DiagonalQuadraticForm(ZZ, [-1,1,1,1,1])

L=150

for j in range(NumberofForms):
    print j,";",factor(Forms[j].determinant()),";",
        Q[j].signature_vector() , ";" , Q[j].disc() , 
        ";" , [Q[j].hasse_invariant(p) for p in [2,..,L] if is_prime(p)]
\end{verbatim}
The output is then a semicolon-separated array of determinants, signatures, discriminants and sequences of Hasse-Witt invariants for primes up to the given bound $L$ (which is in this case $150$).
This list can be sorted easily to give the isometry classes of those forms given by the matrices \verb+Forms[i]+ in the program.
In the above only three forms are given by way of example: the list runs to an index \verb+i=77+.

\section{Invariants for the hypergeometric groups of type $\Orth(3,2)$}\label{sec:Inv}
Below, the quadratic forms associated to the hypergeometric groups (studied above) are listed.
The indices $\alpha$ and $\beta$ are those describing the indices in the characteristic polynomials,
and the first row of a matrix $Q$ for the quadratic form $\q$ is given. 
The Hasse invariants are given at the first few primes, at least as many as to be able to distinguish similarity classes.

\par Where known, the nature of the hypergeometric group is given.
Note that we have examples of arithmetic hypergeometric groups of type $\Orth(3,2)$.

Each quadratic form has been normalised to have discriminant $+1$, and since all forms are of the same signature, each similarity class is given by its sequence of Hasse-Witt invariants.
Thus if two forms fall into the same `Hasse-Witt class' then they are similar: the horizontal lines in the table separate these classes.
(See the discussion after the proof of Lemma~\ref{lem:HasseSymbolInvariant}.)
Where it is known that the hypergeometric groups are arithmetic, isometry of the quadratic forms in this sense implies commensurability of the hypergeometric groups.

\begin{thm}
Wherever two groups in the following tables are listed as arithmetic, they are commensurable if their quadratic forms lie in the same similarity class.
\end{thm}
We see that the hypergeometric groups from our list, that are known to be arithmetic, split into three commensurability classes.

{\begin{center}
\scriptsize\renewcommand{\arraystretch}{2.5}
\begin{longtable}{|l|l|l|l|c|c|}
\hline
  $\alpha$, $\beta$  & 1st row of $Q$   &  Hasse invariants  & Nature  \\                                                                                                                                                        
\hline
\hline
 \pbox{20cm}{ $\alpha=(0,0,0,0,0)$                                                      \\ $\beta=\big(\frac{1}{2}, \frac{1}{6},\frac{1}{6},\frac{5}{6},\frac{5}{6} \big)$        } & $(  57,  39,-7,-57,-71)            $ & $(-1, 1, 1, 1, 1)$ &  Arithmetic \cite{Si15O}  \\
 \pbox{20cm}{ $\alpha=(0,0,0,0,0)$                                                      \\ $\beta=\big(\frac{1}{2},\frac{1}{4},\frac{1}{4},\frac{3}{4},\frac{3}{4}\big)$          } & $( 17,  7,-15,-25,  17)           $ &   &  Arithmetic \cite{Si15O}    \\
 \pbox{20cm}{ $\alpha=\big( 0,0,0,\frac{1}{6},\frac{5}{6}\big)$                         \\ $\beta=\big(\frac{1}{2}, \frac{1}{4},\frac{1}{4},\frac{3}{4},\frac{3}{4} \big)$        } & $(  37,  11,-35,-37,  37)          $ &   &  Arithmetic \cite{BS15} \\
 \pbox{20cm}{ $\alpha=\big( 0,0,0,\frac{1}{4},\frac{3}{4}\big)$                         \\ $\beta=\big(\frac{1}{2}, \frac{1}{6},\frac{1}{6},\frac{5}{6},\frac{5}{6} \big)$        } & $(-7,-9,-7,  7,  25)               $ &   &  Arithmetic \cite{BS15}    \\
 \pbox{20cm}{ $\alpha=\big( 0,0,0,\frac{1}{3},\frac{2}{3}\big)$                         \\ $\beta=\big(\frac{1}{2},\frac{1}{10},\frac{3}{10},\frac{7}{10}, \frac{9}{10} \big)$    } & $(-3,-5,-3,  3,  5)                $ &   &  Arithmetic \cite{Ve14}   \\
 \pbox{20cm}{ $\alpha=\big( 0,0,0,\frac{1}{3},\frac{2}{3}\big)$                         \\ $\beta=\big(\frac{1}{2}, \frac{1}{6},\frac{1}{6},\frac{5}{6},\frac{5}{6} \big)$        } & $(-3,-3,-1,  3,  7)                $ &   &  Arithmetic \cite{BS15}   \\
 \pbox{20cm}{ $\alpha=\big( 0,0,0,\frac{1}{3},\frac{2}{3}\big)$                         \\ $\beta=\big(\frac{1}{2},\frac{1}{4},\frac{1}{4},\frac{3}{4},\frac{3}{4}\big)$          } & $(  5,-5,-3,  11,  5)              $ &   &  Arithmetic \cite{BS15}  \\
 \pbox{20cm}{ $\alpha=\big( 0,0,0,\frac{1}{3},\frac{2}{3}\big)$                         \\ $\beta=\big(\frac{1}{2},\frac{1}{4},\frac{3}{4},\frac{1}{6},\frac{5}{6}\big)$          } & $(-1,-3,-1,  5,  7)                $ &  &  Arithmetic \cite{BS15}\\
 \pbox{20cm}{ $\alpha=\big( 0,\frac{1}{10},\frac{3}{10},\frac{7}{10},\frac{9}{10}\big)$ \\ $\beta=\big(\frac{1}{2},\frac{1}{12},\frac{5}{12},\frac{7}{12}, \frac{11}{12} \big)$   } & $(  21,  19,  11,-1,-9)            $ &  &  Arithmetic \cite{BS15}   \\
 \pbox{20cm}{ $\alpha=\big( 0,\frac{1}{8},\frac{3}{8},\frac{5}{8},\frac{7}{8}\big)$     \\ $\beta=\big(\frac{1}{2},\frac{1}{12},\frac{5}{12},\frac{7}{12}, \frac{11}{12} \big)$   } & $(  1,  3,  1,-1,  1)              $ &   &  Arithmetic \cite{BS15}   \\
 \pbox{20cm}{ $\alpha=\big( 0,\frac{1}{6},\frac{1}{6},\frac{5}{6},\frac{5}{6}\big)$     \\ $\beta=\big(\frac{1}{2},\frac{1}{12},\frac{5}{12},\frac{7}{12}, \frac{11}{12} \big)$   } & $(  19,  17,  10,-1,-8)            $ &   &  Arithmetic \cite{BS15}    \\
 \pbox{20cm}{ $\alpha=\big( 0,\frac{1}{6},\frac{1}{6},\frac{5}{6},\frac{5}{6}\big)$     \\ $\beta=\big(\frac{1}{2},\frac{1}{8},\frac{3}{8},\frac{5}{8}, \frac{7}{8} \big)$        } & $(  73,  53,  1,-55,-71)           $ &   &  Arithmetic \cite{BS15}   \\
 \pbox{20cm}{ $\alpha=\big( 0,\frac{1}{6},\frac{1}{6},\frac{5}{6},\frac{5}{6}\big)$     \\ $\beta=\big(\frac{1}{2},\frac{1}{3},\frac{1}{3},\frac{2}{3},\frac{2}{3}\big)$          } & $(  8, 0,-10, 0,  26)              $ &   &  Arithmetic \cite{BS15}    \\
 \pbox{20cm}{ $\alpha=\big( 0,\frac{1}{5},\frac{2}{5},\frac{3}{5},\frac{4}{5}\big)$     \\ $\beta=\big(\frac{1}{2},\frac{1}{12},\frac{5}{12},\frac{7}{12}, \frac{11}{12} \big)$   } & $(-3,  3,-1,-1,  3)                $ &  &  Arithmetic \cite{Ve14}     \\
 \pbox{20cm}{ $\alpha=\big( 0,\frac{1}{5},\frac{2}{5},\frac{3}{5},\frac{4}{5}\big)$     \\ $\beta=\big(\frac{1}{2},\frac{1}{3},\frac{1}{3},\frac{2}{3},\frac{2}{3}\big)$          } & $(  29,-25,  11,  11,-25)          $ &   &  Arithmetic \cite{BS15}   \\
 \pbox{20cm}{ $\alpha=\big( 0,\frac{1}{4},\frac{1}{4},\frac{3}{4},\frac{3}{4}\big)$     \\ $\beta=\big(\frac{1}{2},\frac{1}{12},\frac{5}{12},\frac{7}{12}, \frac{11}{12} \big)$   } & $(  9,  15,  9,-9,  9)             $ &  &  Arithmetic \cite{Ve14}   \\
 \pbox{20cm}{ $\alpha=\big( 0,\frac{1}{4},\frac{1}{4},\frac{3}{4},\frac{3}{4}\big)$     \\ $\beta=\big(\frac{1}{2},\frac{1}{10},\frac{3}{10},\frac{7}{10}, \frac{9}{10} \big)$    } & $(-3,  3,  5,-5,-3)                $ &   &  Arithmetic \cite{Ve14}      \\
 \pbox{20cm}{ $\alpha=\big( 0,\frac{1}{4},\frac{1}{4},\frac{3}{4},\frac{3}{4}\big)$     \\ $\beta=\big(\frac{1}{2},\frac{1}{8},\frac{3}{8},\frac{5}{8}, \frac{7}{8} \big)$        } & $(  1,  3,  1,-5,  1)              $ & &  Arithmetic \cite{BS15}  \\
 \pbox{20cm}{ $\alpha=\big( 0,\frac{1}{4},\frac{1}{4},\frac{3}{4},\frac{3}{4}\big)$     \\ $\beta=\big(\frac{1}{2},\frac{1}{5},\frac{2}{5},\frac{3}{5},\frac{4}{5} \big)$         } & $(  21,  11,-19,-29,  21)          $ &  &  Arithmetic \cite{Ve14}     \\
 \pbox{20cm}{ $\alpha=\big( 0,\frac{1}{4},\frac{1}{4},\frac{3}{4},\frac{3}{4}\big)$     \\ $\beta=\big(\frac{1}{2}, \frac{1}{3},\frac{1}{3},\frac{2}{3},\frac{2}{3} \big)$        } & $(  73,-17,-71,  55,  73)          $   & &  Arithmetic \cite{BS15}   \\
 \pbox{20cm}{ $\alpha=\big( 0,\frac{1}{4},\frac{3}{4},\frac{1}{6},\frac{5}{6}\big)$     \\ $\beta=\big(\frac{1}{2},\frac{1}{8},\frac{3}{8},\frac{5}{8}, \frac{7}{8} \big)$        } & $(  13,  11,  1,-13,-11)           $ &   &  Arithmetic \cite{BS15}     \\
 \pbox{20cm}{ $\alpha=\big( 0,\frac{1}{4},\frac{3}{4},\frac{1}{6},\frac{5}{6}\big)$     \\ $\beta=\big(\frac{1}{2}, \frac{1}{3},\frac{1}{3},\frac{2}{3},\frac{2}{3} \big)$        } & $(  39,-6,-42,  21,  66)           $ &  &  Arithmetic \cite{BS15}   \\
 \pbox{20cm}{ $\alpha=\big( 0,\frac{1}{3},\frac{2}{3},\frac{1}{6},\frac{5}{6}\big)$     \\ $\beta=\big(\frac{1}{2},\frac{1}{12},\frac{5}{12},\frac{7}{12}, \frac{11}{12} \big)$   } & $(  1,  2,  1,-1,  1)              $ &   &  Arithmetic \cite{BS15}     \\
 \pbox{20cm}{ $\alpha=\big( 0,\frac{1}{3},\frac{2}{3},\frac{1}{6},\frac{5}{6}\big)$     \\ $\beta=\big(\frac{1}{2},\frac{1}{8},\frac{3}{8},\frac{5}{8}, \frac{7}{8} \big)$        } & $(  1,  5,  1,-7,  1)              $ &   &  Arithmetic \cite{BS15} \\
 \pbox{20cm}{ $\alpha=\big( 0,\frac{1}{3},\frac{2}{3},\frac{1}{6},\frac{5}{6}\big)$     \\ $\beta=\big(\frac{1}{2}, \frac{1}{4},\frac{1}{4},\frac{3}{4},\frac{3}{4} \big)$        } & $(  1,-7,  1,  17,  1)             $ &   &  Arithmetic \cite{BS15}    \\
 \pbox{20cm}{ $\alpha=\big( 0,\frac{1}{3},\frac{1}{3},\frac{2}{3},\frac{2}{3}\big)$     \\ $\beta=\big(\frac{1}{2},\frac{1}{12},\frac{5}{12},\frac{7}{12}, \frac{11}{12} \big)$   } & $(-1,  1, 0,-1,  2)                $ & &  Arithmetic \cite{Ve14} \\
 \pbox{20cm}{ $\alpha=\big( 0,\frac{1}{3},\frac{1}{3},\frac{2}{3},\frac{2}{3}\big)$     \\ $\beta=\big(\frac{1}{2},\frac{1}{10},\frac{3}{10},\frac{7}{10}, \frac{9}{10} \big)$    } & $(-3,  1,  3,-3,-1)                $ &  &  Arithmetic \cite{Ve14}      \\
 \pbox{20cm}{ $\alpha=\big( 0,\frac{1}{3},\frac{1}{3},\frac{2}{3},\frac{2}{3}\big)$     \\ $\beta=\big(\frac{1}{2},\frac{1}{8},\frac{3}{8},\frac{5}{8}, \frac{7}{8} \big)$        } & $(-7,  5,  1,-7,  9)               $ &   &  Arithmetic \cite{BS15}  \\
 \pbox{20cm}{ $\alpha=\big( 0,\frac{1}{3},\frac{1}{3},\frac{2}{3},\frac{2}{3}\big)$     \\ $\beta=\big(\frac{1}{2}, \frac{1}{6},\frac{1}{6},\frac{5}{6},\frac{5}{6} \big)$        } & $( 0, 0,  2, 0,-2)                 $ &  &  Arithmetic \cite{Ve14}     \\
 \pbox{20cm}{ $\alpha=\big( 0,\frac{1}{3},\frac{1}{3},\frac{2}{3},\frac{2}{3}\big)$     \\ $\beta=\big(\frac{1}{2},\frac{1}{5},\frac{2}{5},\frac{3}{5},\frac{4}{5} \big)$         } & $(-11,  9,-1,-11,  19)             $ &   &  Arithmetic \cite{Ve14}       \\
 \pbox{20cm}{ $\alpha=\big( 0,\frac{1}{3},\frac{1}{3},\frac{2}{3},\frac{2}{3}\big)$     \\ $\beta=\big(\frac{1}{2},\frac{1}{4},\frac{1}{4},\frac{3}{4},\frac{3}{4}\big)$          } & $(-7,  1,  9,-7,-7)                $ &   &  Arithmetic \cite{Ve14}     \\
 \pbox{20cm}{ $\alpha=\big( 0,\frac{1}{3},\frac{1}{3},\frac{2}{3},\frac{2}{3}\big)$     \\ $\beta=\big(\frac{1}{2},\frac{1}{4},\frac{3}{4},\frac{1}{6},\frac{5}{6}\big)$          } & $(-1, 0,  2,-1,-2)                 $ &  &  Arithmetic \cite{BS15}     \\
 \pbox{20cm}{ $\alpha=\big( 0,\frac{1}{3},\frac{2}{3},\frac{1}{4},\frac{3}{4}\big)$     \\ $\beta=\big(\frac{1}{2},\frac{1}{8},\frac{3}{8},\frac{5}{8}, \frac{7}{8} \big)$        } & $(-3,  3,  1,-5,  5)               $ &   &  Arithmetic \cite{BS15}   \\
\hline
\pbox{20cm}{ $\alpha=\big( 0,\frac{1}{4},\frac{3}{4},\frac{1}{6},\frac{5}{6}\big)$     \\ $\beta=\big(\frac{1}{2},\frac{1}{12},\frac{5}{12},\frac{7}{12}, \frac{11}{12} \big)$   } & $(  5,  5,  3,-1,-1)               $ &  $(1, -1, 1, 1, 1)$ &  Arithmetic \cite{BS15}     \\
\pbox{20cm}{ $\alpha=\big( 0,\frac{1}{4},\frac{3}{4},\frac{1}{6},\frac{5}{6}\big)$     \\ $\beta=\big(\frac{1}{2},\frac{1}{10},\frac{3}{10},\frac{7}{10}, \frac{9}{10} \big)$    } & $(  5,  7,  5,-5,-7)               $ &    &  Arithmetic \cite{BS15}      \\
\pbox{20cm}{ $\alpha=\big( 0,\frac{1}{3},\frac{2}{3},\frac{1}{4},\frac{3}{4}\big)$     \\ $\beta=\big(\frac{1}{2},\frac{1}{12},\frac{5}{12},\frac{7}{12}, \frac{11}{12} \big)$   } & $(-3,  9,  3,-3,  15)              $ & &  Arithmetic \cite{BS15}   \\
\hline
\pbox{20cm}{ $\alpha=\big( 0,\frac{1}{3},\frac{2}{3},\frac{1}{4},\frac{3}{4}\big)$     \\ $\beta=\big(\frac{1}{2},\frac{1}{5},\frac{2}{5},\frac{3}{5},\frac{4}{5} \big)$         } & $(-3,  7,-3,-13,  17)              $ & $(1, 1, -1, 1, 1)$   &  Arithmetic \cite{Ve14}     \\
\hline
\end{longtable}
\end{center}
}

In the following table we list quadratic forms of signature $(3,2)$, the natures of whose associated hypergeometric groups are not known.
One may see that many of these forms are similar to forms in the above table and so if their groups were known to be arithmetic then they would be commensurable to groups above.

\par If all these groups turned out to be arithmetic then there would be four commensurability classes in total for groups in $\Orth(3,2)$.

{\begin{center}
\scriptsize\renewcommand{\arraystretch}{2.5}
\begin{longtable}{|l|l|l|l|c|c|}
\hline
$\alpha$, $\beta$  & 1st row of $Q$ &  Hasse invariants & Nature  \\                                                                                                                                                            
\hline
\hline
  \pbox{20cm}{ $\alpha=\big( 0,0,0,\frac{1}{6},\frac{5}{6}\big)$                       \\  $\beta=\big(\frac{1}{2},\frac{1}{5},\frac{2}{5},\frac{3}{5},\frac{4}{5} \big)$         } & $(  53,  19,-43,-53,  29)           $  &   $(-1, -1, -1, 1, 1)$ &  ??   \\
\hline
 \pbox{20cm}{ $\alpha=(0,0,0,0,0)$                                                    \\  $\beta=\big(\frac{1}{2},\frac{1}{12},\frac{5}{12},\frac{7}{12}, \frac{11}{12} \big)$   } & $(  257,  223,  129,-1,-127)         $   &  $(-1, 1, 1, 1, 1)$   &  ??           \\
 \pbox{20cm}{ $\alpha=(0,0,0,0,0)$                                                    \\  $\beta=\big(\frac{1}{2},\frac{1}{10},\frac{3}{10},\frac{7}{10}, \frac{9}{10} \big)$     } & $(  141,  115,  45,-45,-115)        $  &     &  ??     \\
 \pbox{20cm}{ $\alpha=(0,0,0,0,0)$                                                    \\  $\beta=\big(\frac{1}{2},\frac{1}{8},\frac{3}{8},\frac{5}{8}, \frac{7}{8} \big)$        } & $(  65,  47,  1,-49,-63)             $  &    &  ??    \\
 \pbox{20cm}{ $\alpha=(0,0,0,0,0)$                                                    \\  $\beta=\big(\frac{1}{2},\frac{1}{5},\frac{2}{5},\frac{3}{5},\frac{4}{5} \big)$          } & $(  445,  195,-355,-605,  445)      $  &     &  ??          \\
 \pbox{20cm}{ $\alpha=(0,0,0,0,0)$                                                    \\  $\beta=\big(\frac{1}{2},\frac{1}{4},\frac{3}{4},\frac{1}{6},\frac{5}{6}\big)$           } & $(  35,  21,-13,-43,-29)            $  &     &  ??  \\
 \pbox{20cm}{ $\alpha=(0,0,0,0,0)$                                                    \\  $\beta=\big(\frac{1}{2},\frac{1}{3},\frac{1}{3},\frac{2}{3},\frac{2}{3}\big)$          } & $(  81,  15,-111,-81,  465)        $  &    &  ??    \\
 \pbox{20cm}{ $\alpha=(0,0,0,0,0)$                                                    \\  $\beta=\big(\frac{1}{2},\frac{1}{3},\frac{2}{3},\frac{1}{6},\frac{5}{6}\big)$           } & $(  265,  151,-119,-329,-119)        $  &   &  ??    \\
 \pbox{20cm}{ $\alpha=(0,0,0,0,0)$                                                    \\  $\beta=\big(\frac{1}{2},\frac{1}{3},\frac{2}{3},\frac{1}{4},\frac{3}{4}\big)$          } & $(  115,  37,-125,-155,  307)        $  &      &  ??    \\
 \pbox{20cm}{ $\alpha=(0,0,0,0,0)$                                                    \\  $\beta=\big(\frac{1}{2},\frac{1}{2},\frac{1}{2},\frac{1}{6},\frac{5}{6}\big)$           } & $(  27,  15,-13,-33,-5)              $  &    &  ??   \\
 \pbox{20cm}{ $\alpha=(0,0,0,0,0)$                                                    \\  $\beta=\big(\frac{1}{2},\frac{1}{2},\frac{1}{2},\frac{1}{4},\frac{3}{4}\big)$          } & $(  11,  3,-13,-13,  43)             $  &    &  ??        \\
 \pbox{20cm}{ $\alpha=(0,0,0,0,0)$                                                    \\  $\beta=\big(\frac{1}{2},\frac{1}{2},\frac{1}{2},\frac{1}{3},\frac{2}{3}\big)$           } & $(  67,  7,-101,-41,  547)           $  &     &  ??     \\
 \pbox{20cm}{ $\alpha=(0,0,0,0,0)$                                                    \\  $\beta=\big(\frac{1}{2},\frac{1}{2},\frac{1}{2},\frac{1}{2},\frac{1}{2}\big)$          } & $(  6,-0,-10,-0,  70)                $  &     &  ??    \\
 \pbox{20cm}{ $\alpha=\big( 0,0,0,\frac{1}{6},\frac{5}{6}\big)$                       \\  $\beta=\big(\frac{1}{2},\frac{1}{3},\frac{1}{3},\frac{2}{3},\frac{2}{3}\big)$           } & $(  13,  1,-17,-5,  55)$  &  &  ??   \\
 \pbox{20cm}{ $\alpha=\big( 0,0,0,\frac{1}{6},\frac{5}{6}\big)$                       \\  $\beta=\big(\frac{1}{2},\frac{1}{3},\frac{2}{3},\frac{1}{4},\frac{3}{4}\big)$           } & $(  23,  5,-25,-19,  47) $  &   &  ??  \\
 \pbox{20cm}{ $\alpha=\big( 0,0,0,\frac{1}{4},\frac{3}{4}\big)$                       \\  $\beta=\big(\frac{1}{2},\frac{1}{3},\frac{1}{3},\frac{2}{3},\frac{2}{3}\big)$           } & $(  273,-33,-303,  111,  561)$  &     &  ??     \\
 \pbox{20cm}{ $\alpha=\big( 0,\frac{1}{6},\frac{1}{6},\frac{5}{6},\frac{5}{6}\big)$   \\  $\beta=\big(\frac{1}{2},\frac{1}{5},\frac{2}{5},\frac{3}{5}, \frac{4}{5} \big)$        } & $(  149,  49,-121,-131,  59)$  &   &  ??    \\
\hline
   \pbox{20cm}{ $\alpha=\big( 0,0,0,\frac{1}{4},\frac{3}{4}\big)$                       \\  $\beta=\big(\frac{1}{2},\frac{1}{2},\frac{1}{2},\frac{1}{3},\frac{2}{3}\big)$          } & $(-168, 30, 192, -114, -228)   $  & $(1, -1, 1, 1, 1)$    &  ??   \\
   \pbox{20cm}{ $\alpha=\big( 0,0,0,\frac{1}{6},\frac{5}{6}\big)$                       \\  $\beta=\big(\frac{1}{2},\frac{1}{2},\frac{1}{2},\frac{1}{3},\frac{2}{3}\big)$          } & $(-10, 0, 28, 0, -116)           $  &    &  ??    \\
 \hline
\end{longtable}
\end{center}
}

\section{Invariants for the hypergeometric groups of type $O(4,1)$}\label{sec:Inv-hyp}

In the tables below we list the same data as in the previous section but this time for quadratic forms of signature $(4,1)$. Here we indicate arithmeticity and thin-ness where known. The groups whose nature is unknown (indicated by `??') could be lattices or not.
Again, as with the case of $\Orth(3,2)$, we have four similarity classes of quadratic forms, and this time we normalise the discriminants to $-1$ (using Lemmas~\ref{lem:NormaliseDiscriminant} and \ref{lem:HasseSymbolInvariant}).

\par Note the following:
\begin{lemma}
Suppose $\Gamma_{i}\leq \SO_{\q_i}(\Z)$ for $i=1,2$, where each $\Gamma_{i}$ is Zariski dense in $\SO_{\q_i}$. Suppose further that there exists a conjugation $\phi\colon \SO_{\q_1} \to \SO_{\q_2}$ such that $\phi(\Gamma_{1})\cap\Gamma_{2}$ is Zariski dense in $\SO_{\q_2}$. Then $\phi$ is defined over $\Q$. 
\end{lemma}

By this lemma, if two thin subgroups of two $\SO_{\q_i}$ can be conjugated to intersect in a Zariski dense subgroup, then the quadratic forms $\q_i$ ($i=1,2$) are isometric, because in odd dimension the $\Q$-isomorphism $\phi$ implies isometry of the two quadratic forms.

\par In the following table, by the preceeding argument, any forms lying in different equivalence classes (with respect to their Hasse-Witt invariants) cannot intersect in a Zariski dense subgroup of $\SO_{\q}$.

{\begin{center}
\scriptsize\renewcommand{\arraystretch}{2.5}
\begin{longtable}{|l|l|l|l|c|}
\hline
  $\alpha$, $\beta$ & 1st row of $Q$  &   Hasse invariants & Nature   \\                                                                                                                                                         

\hline
\hline
\pbox{20cm}{ $\alpha=\big( 0,0,0,\frac{1}{3},\frac{2}{3}\big)$                       \\ $\beta=\big(\frac{1}{2},\frac{1}{12},\frac{5}{12},\frac{7}{12}, \frac{11}{12} \big)$  } & $(7, 5, 3, 1, -5)        $  &   $(1, 1, 1, 1, 1)$     &    Thin \cite{FMS14}     \\
\pbox{20cm}{ $\alpha=\big( 0,0,0,\frac{1}{3},\frac{2}{3}\big)$                       \\  $\beta=\big(\frac{1}{2},\frac{1}{8},\frac{3}{8},\frac{5}{8}, \frac{7}{8} \big)$       } & $(7, 1, -1, 1, -9)       $  &      &    Thin \cite{FMS14}     \\
 \pbox{20cm}{ $\alpha=\big( 0,0,0,\frac{1}{4},\frac{3}{4}\big)$                       \\  $\beta=\big(\frac{1}{2},\frac{1}{10},\frac{3}{10},\frac{7}{10}, \frac{9}{10} \big)$   } & $(19, 13, 3, -3, -13)    $  &       &    Thin \cite{FMS14}      \\
 \pbox{20cm}{ $\alpha=\big( 0,\frac{1}{8},\frac{3}{8},\frac{5}{8},\frac{7}{8}\big)$   \\  $\beta=\big(\frac{1}{2},\frac{1}{5},\frac{2}{5},\frac{3}{5}, \frac{4}{5} \big)$       } & $(19, -11, -1, 9, -21)   $  &      &    Thin \cite{FMS14}   \\  
\hline
\pbox{20cm}{ $\alpha=\big( 0,0,0,\frac{1}{6},\frac{5}{6}\big)$                       \\  $\beta=\big(\frac{1}{2},\frac{1}{10},\frac{3}{10},\frac{7}{10}, \frac{9}{10} \big)$   } & $(67, 53, 19, -19, -53)  $  &  $(-1, -1, 1, 1, 1)$   &    Thin \cite{FMS14}      \\
\pbox{20cm}{ $\alpha=\big( 0,\frac{1}{3},\frac{2}{3},\frac{1}{6},\frac{5}{6}\big)$   \\  $\beta=\big(\frac{1}{2},\frac{1}{5},\frac{2}{5},\frac{3}{5},\frac{4}{5} \big)$        } & $(19, -1, -11, -1, -11)  $    &    &    Thin \cite{FMS14}     \\
\hline
\pbox{20cm}{ $\alpha=\big( 0,\frac{1}{4},\frac{3}{4},\frac{1}{6},\frac{5}{6}\big)$   \\  $\beta=\big(\frac{1}{2},\frac{1}{5},\frac{2}{5},\frac{3}{5},\frac{4}{5} \big)$        } & $(67, 17, -53, -43, 7)   $  &  $(1, -1, -1, 1, 1)$   &    Thin \cite{FMS14}        \\
\hline
\end{longtable}
\end{center}
}

\newpage
In the following table are listed forms of type $\Orth(4,1)$, the natures of whose hypergeometric groups are unknown.
{\begin{center}
\scriptsize\renewcommand{\arraystretch}{2.5}
\begin{longtable}{|l|l|l|l|c|}
\hline
 $\alpha$, $\beta$  & First row of $Q$  & Hasse invariants  & Nature \\                                                                                                                                                                
\hline
\hline 
  \pbox{20cm}{ $\alpha=\big( 0,0,0,\frac{1}{4},\frac{3}{4}\big)$                         \\ $\beta=\big(\frac{1}{2},\frac{1}{12},\frac{5}{12},\frac{7}{12}, \frac{11}{12} \big)$ }  & $(183, 153, 87, 9, -105)         $ & $(-1, -1, 1, 1, 1)$  & ??  \\
  \pbox{20cm}{ $\alpha=\big( 0,0,0,\frac{1}{4},\frac{3}{4}\big)$                         \\ $\beta=\big(\frac{1}{2},\frac{1}{3},\frac{2}{3},\frac{1}{6},\frac{5}{6}\big)$         } & $(183, 57, -105, -87, -105)  $ &   & ??  \\
  \hline
 \pbox{20cm}{ $\alpha=\big( 0,0,0,\frac{1}{3},\frac{2}{3}\big)$                         \\ $\beta=\big(\frac{1}{2},\frac{1}{5},\frac{2}{5},\frac{3}{5},\frac{4}{5} \big)$       }  & $(11, -3, -5, 5, -13)               $ &  $(-1, 1, -1, 1, 1)$ & ??  \\
\hline
 \pbox{20cm}{ $\alpha=\big( 0,0,0,\frac{1}{6},\frac{5}{6}\big)$                         \\ $\beta=\big(\frac{1}{2},\frac{1}{12},\frac{5}{12},\frac{7}{12}, \frac{11}{12} \big)$ }  & $(71, 61, 35, 1, -37)               $ & $(1, 1, 1, 1, 1)$   & ??  \\
 \pbox{20cm}{ $\alpha=\big( 0,0,0,\frac{1}{6},\frac{5}{6}\big)$                         \\ $\beta=\big(\frac{1}{2},\frac{1}{8},\frac{3}{8},\frac{5}{8}, \frac{7}{8} \big)$      }  & $(71, 49, -1, -47, -73)             $ &     & ??  \\
 \pbox{20cm}{ $\alpha=\big( 0,0,0,\frac{1}{4},\frac{3}{4}\big)$                         \\ $\beta=\big(\frac{1}{2},\frac{1}{8},\frac{3}{8},\frac{5}{8}, \frac{7}{8} \big)$      }  & $(15, 9, -1, -7, -17)               $ &   & ??  \\
 \pbox{20cm}{ $\alpha=\big( 0,0,0,\frac{1}{4},\frac{3}{4}\big)$                         \\ $\beta=\big(\frac{1}{2},\frac{1}{5},\frac{2}{5},\frac{3}{5},\frac{4}{5} \big)$       }  & $(27, 5, -21, -11, -5)              $ &     & ??  \\
 \pbox{20cm}{ $\alpha=\big( 0,0,0,\frac{1}{3},\frac{2}{3}\big)$                         \\ $\beta=\big(\frac{1}{2},\frac{1}{2},\frac{1}{2},\frac{1}{6},\frac{5}{6}\big)$         } & $(6, 0, -2, 0, -10)                 $ &    & ??  \\
 \pbox{20cm}{ $\alpha=\big( 0,0,0,\frac{1}{3},\frac{2}{3}\big)$                         \\ $\beta=\big(\frac{1}{2},\frac{1}{2},\frac{1}{2},\frac{1}{4},\frac{3}{4}\big)$        }  & $(16, -6, -8, 10, -16)             $ &     & ??  \\
 \pbox{20cm}{ $\alpha=\big( 0,\frac{1}{10},\frac{3}{10},\frac{7}{10},\frac{9}{10}\big)$ \\ $\beta=\big(\frac{1}{2},\frac{1}{5},\frac{2}{5},\frac{3}{5},\frac{4}{5} \big)$       }  & $(15, 0, -10, 0, -10)            $ &    & ??  \\
\hline
\end{longtable}
\end{center}
}

\section{Invariants for the hypergeometric groups of finite type}\label{sec:Inv-ft}

The forms listed in this section correspond to hypergeometric groups of finite type. They preserve a positive definite quadratic form and have discriminant $+1$ over $\Q$.
{
\begin{center}
\scriptsize\renewcommand{\arraystretch}{2.5}
\begin{longtable}{|l|l|l|l|l|c|}
\hline
  $\alpha$, $\beta$                                                                                                                                                                & 1st row of $Q$          &   Hasse invariants      & Structure & Order  \\
\hline
\hline
 \pbox{20cm}{ $\alpha=\big( 0,\frac{1}{5},\frac{2}{5},\frac{3}{5},\frac{4}{5}\big)$  \\ $\beta=\big(\frac{1}{2},\frac{1}{10},\frac{3}{10},\frac{7}{10}, \frac{9}{10} \big)$    }   & $(1, 0, 0, 0, 0)    $    &   $(1, 1, 1, 1, 1)$   & $\Z_{2} \times (\Z_{2}^{4} \rtimes \Z_{5})$  & 160 \\
 \pbox{20cm}{ $\alpha=\big( 0,\frac{1}{5},\frac{2}{5},\frac{3}{5},\frac{4}{5}\big)$  \\ $\beta=\big(\frac{1}{2},\frac{1}{8},\frac{3}{8},\frac{5}{8}, \frac{7}{8} \big)$        }   & $(5, -3, 1, 1, -3)  $ &      & $(\Z_{2}^{4} \rtimes A_{5}) \rtimes \Z_{2} $&1920     \\
 \pbox{20cm}{ $\alpha=\big( 0,\frac{1}{3},\frac{2}{3},\frac{1}{4},\frac{3}{4}\big)$  \\ $\beta=\big(\frac{1}{2},\frac{1}{10},\frac{3}{10},\frac{7}{10}, \frac{9}{10} \big)$    }   & $(5, -1, -3, 3, 1)  $  &        & $\Z_{2} \times ((\Z_{2}^{4} \rtimes A_{5}) \rtimes \Z_{2}) $ & 3840    \\
\hline
 \pbox{20cm}{ $\alpha=\big( 0,\frac{1}{3},\frac{2}{3},\frac{1}{6},\frac{5}{6}\big)$  \\ $\beta=\big(\frac{1}{2},\frac{1}{10},\frac{3}{10},\frac{7}{10}, \frac{9}{10} \big)$    }   & $(5, 1, -1, 1, -1)  $  &   $(-1, -1, 1, 1, 1)$   & $\Z_{2} \times S_{6} $ & 1440    \\
\hline
\end{longtable}
\end{center}
}

In this table, $\Z^{4}_{2} = \Z_{2}\times \Z_{2} \times \Z_{2} \times \Z_{2}$, and $S_{n}$ (resp.~$A_{n}$) denotes the symmetric group (resp.~alternating group) on $n$ elements.
Here we were able to use \gap\ \cite{GAP4} to determine the structure of each hypergeometric group.
Given the generating matrices $A$ and $B$ of the group $\Gamma(f,g)$,  \gap\ computes the group generated by these matrices and so obtains a representation of $\Gamma(f,g)$ (cf.~\S\ref{sec:monodromygroups}).
Since the group is finite this computation ends in finite time.
We used the following commands to obtain the structures, having specified $A$ and $B$:

\begin{verbatim}
gap> G:=Group(A,B);
gap> Size(G);
gap> StructureDescription(G);
\end{verbatim}
 
\par One sees from the fourth column of the above table that the groups in the first equivalence class have similar sub-structures.
Though we have not computed explicit conjugations between the three forms in this class, one would expect to be able to find a conjugation sending these common subgroups to one another.

\section*{Acknowledgements}
The authors are grateful to the Max Planck Institute for Mathematics in Bonn, where much of the requisite discussion took place, for its support and hospitality.
JB would like to thank Gerhard Hiss and Klaus Lux for many interesting conversations during his visit to Aachen, and thanks the Mathematical Institute of the University of Bern for its hospitality.
SS is supported by the \textsc{inspire} Fellowship from the Department of Science and Technology, India.
ST is supported by the SNF (Switzerland) grant number \texttt{PP00P2\_157583}.

\nocite{}
\bibliographystyle{amsalpha}
\bibliography{cchg}

\end{document}